\batchmode
\makeatletter
\def\input@path{{\string"D:/Weberszpil/Calculo Fracional/Artigos 2013/Leibniz Rule/CNSNS/CNSNS-2014-Ressub3/\string"/}}
\makeatother
\documentclass[twoside,english,aps]{elsarticle}
\usepackage[LGR,T1]{fontenc}
\usepackage[latin9]{inputenc}
\usepackage{geometry}
\usepackage{textcomp}
\usepackage{amsthm}
\usepackage{amstext}
\usepackage{amssymb}
\usepackage{esint}

\makeatletter

\DeclareRobustCommand{\greektext}{%
  \fontencoding{LGR}\selectfont\def\encodingdefault{LGR}}
\DeclareRobustCommand{\textgreek}[1]{\leavevmode{\greektext #1}}
\DeclareFontEncoding{LGR}{}{}
\DeclareTextSymbol{\~}{LGR}{126}
\newcommand{\lyxmathsym}[1]{\ifmmode\begingroup\def\b@ld{bold}
  \text{\ifx\math@version\b@ld\bfseries\fi#1}\endgroup\else#1\fi}

\theoremstyle{plain}
\newtheorem{thm}{\protect\theoremname}
\theoremstyle{plain}
\newtheorem{lem}[thm]{\protect\lemmaname}
\ifx\proof\undefined
\newenvironment{proof}[1][\protect\proofname]{\par
\normalfont\topsep6\p@\@plus6\p@\relax
\trivlist
\itemindent\parindent
\item[\hskip\labelsep
\scshape
#1]\ignorespaces
}{%
\endtrivlist\@endpefalse
}
\providecommand{\proofname}{Proof}
\fi

\makeatother

\usepackage{babel}
\providecommand{\lemmaname}{Lemma}
\providecommand{\theoremname}{Theorem}

\begin{document}

\title{Validity of the fractional Leibniz rule on a coarse-grained medium
yields a modified fractional chain rule}

\author{J. Weberszpil\corref{cor1}}

\ead{josewebe@gmail.com}

\cortext[cor1]{ Universidade Federal Rural do Rio de Janeiro, UFRRJ-IM/DTL\\
 Av. Governador Roberto Silveira s/n- Nova Iguaçú, Rio de Janeiro,
Brasil, 695014.}
\begin{abstract}
In this short communication, we show that the validity of the Leibniz
rule for a fractional derivative on a coarse-grained medium brings
about a modified chain rule, in agreement with alternative versions
of fractional calculus. We compare our results to those of a recent
article on this matter.\end{abstract}
\maketitle

\section{Introduction}

Fractional calculus (FC) is a mathematical tool which has a wide range
of applications to mathematical modelling in physical sciences, biological
and biomedical engineering, control systems and dynamic modelling
of soils, among many others \cite{Ours,Jumarie}. It has a history
of over 300 years and an extensive bibliography has been produced
on this subject matter. Complex systems whose dynamical behaviour
is described by so-called anomalous functions, with fractional powers
decreasing behaviour are well described by the FC machinery. There
are several approaches to FC, depending on the applications envisaged.

In a recent article \cite{Tarasov}, it is argued that violation of
the Leibniz rule is a characteristic property of non-integer order
derivatives, and that non-integer order derivatives satisfying the
Leibniz rule must be of entire order $\alpha=1$. In the same article,
the author suggests that the result, which has been proven for a function
$f$ of class $C^{2}$, $f\in C^{2}(U)$, would hold also for any
formulation of FC applied to functions that are not necessarily differentiable.

In certain physical models of our real world, one deals with observables
that are not necessarily differentiable functions, and may have fractal
characteristics (so-called 'coarse-grained' spaces \cite{Gell-Mann}),
see e.g.\inputencoding{latin1}{ }\inputencoding{latin9}\cite{Stillinger},
\cite{Nottale} or \cite{G. }.

A continuous, nowhere integer-differentiable function, necessarily
exhibits random-like or pseudo-random features, which links this story
to extensions of the theory of stochastic differential equations to
describe stochastic dynamics driven by fractional Brownian motion
\cite{Jumarie1,Jumarie-Lagrang Fract,Jumarie2}. For an interesting
effort in building up a solid geometry and field theory in fractional
spaces, see \citep{G. ,G. - Arxiv1106.5787,G. - Arxiv1107.5041,G. - Arxiv1305.3497}.

A good mathematical framework for fractional derivative operators
is that of Hölder spaces $H^{\lambda}$ \cite{Samko-Kilbas-Marichev},
which finds abundant applications even in the stock market \cite{Stock Market}.

In this letter, we follow in the footsteps of \cite{Tarasov} and
start with a fractional differential operator $D^{\alpha}$ of order
$\alpha$, which is assumed to satisfy the Leibniz rule, then derive
a fractional chain rule of sorts for $D^{\alpha}$. Throughout we
deal with identities just as in \cite{Tarasov}, and not with $L^{2}-$type
almost-everywhere equalities throughout.

\section{Hölder space and coarse-grained media}

According to Samko et al. \citep{Samko-Kilbas-Marichev}, let $\Omega=(a,b),$
$-\infty<a<b<\infty$ , so $\Omega$ may be a finite interval, a half-line
or the whole line. First let $\Omega$ be a finite interval. A function
$F(x)$ is said to satisfy the Hölder condition of order $\lambda$
on $\Omega$ if 
\begin{equation}
|F(x_{1})-F(x_{2})|\le A|x_{1}-x_{2}|^{\alpha},\label{eq:Holder condition}
\end{equation}
for any $x_{1}$, $x_{2}\in\Omega,$ where $A$ is a constant and
$\alpha$ is the Hölder exponent. If the function $f(x)$ satisfies
the Hölder condition it is continuous on $\Omega$ \citep{Samko-Kilbas-Marichev},
that is, if the Hölder coefficient is merely bounded on compact subsets
of $\lyxmathsym{\textgreek{W}}$, then the function $F$ is said to
be locally Hölder-continuous with exponent $\lyxmathsym{\textgreek{a}}$
in $\lyxmathsym{\textgreek{W}}$.

Now to define the Hölder space, let again $\Omega$ be a finite interval.
We denote by $H^{\lambda}=H^{\lambda}(\Omega)$ the space of all functions
which in general are complex valued, and satisfying the Hölder condition
of a fixed order $\lambda$ on $\Omega.$ This is a locally convex
topological vector space.

Hölder space and nowhere differentiable functions are related. An
immediate example of this is the Weierstrass function that is nowhere
integer-differentiable \citep{Kolwankar- Nowhere}. The Weierstrass
function may be written as 
\begin{equation}
W_{\alpha}(x)=\sum_{n=0}^{\infty}b^{-n\alpha}\cos(b^{n}x)
\end{equation}
for some $0<\lyxmathsym{\textgreek{a}}<1$. Then $W_{\alpha}(x)$
is Hölder-continuous of exponent $\lyxmathsym{\textgreek{a}}$, which
is to say that there is a constant C such that 
\begin{equation}
|W_{\alpha}(x)-W_{\alpha}(y)|\le C|x-y|^{\alpha}
\end{equation}
for all x and y. Moreover, $W_{1}$ is Hölder-continuous of all orders
$\lyxmathsym{\textgreek{a}}<1$ but not Lipschitz continuous.

\section{Main results}

Below we derive, out of the Leibniz rule, a fractional chain rule
of sorts for compositions $f\circ w$ where $f$ is $C^{2}$ and $w$
is not necessarily $C^{1}$, similarly as in \citep{Jumarie}. Here,
we assume the same premisses for the conditions of operational linearity,
the fractional derivative of a constant being zero and the validity
of the fractional Leibniz rule as in \citep{Tarasov}.

Just as a locally $L^{1}$ function $f$ is a measurable function
which is $L^{1}$ at every compact interval, we call a function locally
Hölder of exponent $\alpha$ if and only if $f$ is Hölder of exponent
$\alpha$ on every finite interval $[a,b].$ One such example is that
of $f\in C^{1}(\mathbb{R})$, taking $\alpha=1$.
\begin{lem}
Let $f,g$ be functions that are locally Hölder-continuous of exponent
$\alpha$, namely, satisfying the Hölder condition of exponent $\alpha$
on every compact interval $[a,b]$ of their domain. Then so is $fg.$\end{lem}
\begin{proof}
It is entirely analogous to that of the formula for the derivative
of a product. $\blacksquare$\end{proof}
\begin{thm}
Let $f\in C^{2}(I)$, where $I,J\subset{\mathbb{R}}$, and let $w:J\to I$
be a Hölder-continuous function of exponent $0<\alpha<1$, where $J\subset{\mathbb{R}}$
is another interval. Consider a fractional derivative $D^{\alpha}$
of order $\alpha$, satisfying the Leibniz rule, whose domain includes
all locally Hölder-continuous functions of order $\alpha$. Then the
following statement holds: 
\begin{equation}
D_{x}^{\alpha}(f\circ w)=D_{w}^{\alpha}f(w(x))(D_{x}^{\alpha}w(x)).
\end{equation}
\end{thm}
\begin{proof}
Clearly, $D_{x}^{\alpha}1=0$ \citep{Tarasov}, since the Leibniz
rule is assumed. Let $x_{0}\in J$, and let $t_{0}=w(x_{0})$ and
$t=w(x)$. Indeed, by repeated integration by parts one obtains the
following identity: 
\begin{equation}
f(t)-f(t_{0})=f'(t_{0})(t-t_{0})+g_{2}(t)(t-t_{0})^{2},\label{eq:eqHadamard}
\end{equation}
where $g_{2}(t)=\int_{0}^{1}(1-s)\, f''(t_{0}+s(t-t_{0}))ds$. Note
that,

\begin{equation}
g_{2}(t)=\left\{ \begin{array}{ccc}
\frac{f(t)-f(t_{0})}{(t-t_{0})^{2}}-\frac{f'(t_{0})}{(t-t_{0})}= & \int_{t_{0}}^{t}\frac{(t-\tau)f''(\tau)}{(t-t_{0})^{2}}d\tau, & t\neq t_{0}\\
2^{-1}f''(x), &  & t=t_{0}
\end{array}\right.,\label{eq:g2}
\end{equation}
and using Taylor\textquoteright{}s theorem 
\begin{equation}
g_{2}(t)=\left\{ \begin{array}{cc}
\frac{2^{-1}f''(t_{0})(t-t_{0})^{2}+o((t-t_{0})^{2})}{(t-t_{0})^{2}}, & t\neq t_{0}\\
2^{-1}f''(x), & t=t_{0}
\end{array}\right..
\end{equation}

From which it follows directly that $g_{2}(t)(t\lyxmathsym{\textminus}t_{0})$
and $g_{2}(t)(t\lyxmathsym{\textminus}t_{0})^{2}$ belong, respectiely,
to $C^{1}(I)$ and $C^{2}(I)$.

Now, from eq.(\ref{eq:eqHadamard}) one obtains the following identity:

\begin{equation}
f(w(x))=f(w(x_{0}))+f'(w(x_{0}))(w(x)-w(x_{0}))+[g_{2}(w(x))(w(x)-w(x_{0}))][w(x)-w(x_{0})].\label{eq:formfin}
\end{equation}

Applying the Leibniz rule to the product $[g_{2}(w(t))(w(t)-w(x_{0}))]\cdotp(w(x)-w(x_{0})$
yields 
\begin{equation}
D_{x}^{\alpha}([g_{2}(w(x))(w(x)-w(x_{0}))]\cdotp(w(x)-w(x_{0}))=(w(x)-w(x_{0}))H(x),
\end{equation}
for a suitable function $H(x)$. It follows that, for $x=x_{0}$,

\begin{equation}
D_{x}^{\alpha}(g_{2}(w(x))(w(x)-w(x_{0}))^{2})|_{x=x_{0}}=0.
\end{equation}

Thus, on applying $D_{x}^{\alpha}$ to (\ref{eq:formfin}) and evaluating
on $x_{0}\in J$, one obtains 
\begin{equation}
D_{x}^{\alpha}(f\circ w)(x_{0})=f'(w(x_{0}))(D_{x}^{\alpha}w)(x_{0}),\label{eq:chainrule}
\end{equation}
which settles the Theorem. $\blacksquare$ 
\end{proof}
We remark that every function $w$ in the domain of $D^{\alpha}$
is allowed in our formula, provided $f\circ w$ also lies in its domain.

\subsection*{Non-Integer Differentiable Functions}

Consider now that $w$ is of class $C^{1}$ and $f$ is Hölder-continuous,
then for Riemann-Liouville fractional derivatives and even for the
Caputo definition, the scale property holds through 
\begin{equation}
D_{x}^{\alpha}f(\lambda x)=\lambda^{\alpha}D_{w}^{\alpha}f(w);\qquad w=\lambda x.\label{eq:11}
\end{equation}

But, considering the relation 
\begin{equation}
\lambda^{\alpha}=[D_{x}^{1}(\lambda x)]^{\alpha}\label{eq:12}
\end{equation}
then, eq.(\ref{eq:11}) can be rewritten as 
\begin{equation}
D_{x}^{\alpha}f(\lambda x)=[D_{x}^{1}(\lambda x)]^{\alpha}D_{w}^{\alpha}f(w)\label{eq:13}
\end{equation}
or 
\begin{equation}
D_{x}^{\alpha}(f\circ w)=[(D_{w}^{\alpha}f)\circ w]\cdotp(w')^{\alpha}.\label{eq:formjumarie}
\end{equation}

Eq. (\ref{eq:formjumarie}) turns out to be the same chain rule that
is valid for nondifferentiable functions in the alternative versions
of FC \cite{Jumarie}, considering $w=\lambda x$ as differentiable
while $f(w)$ is nondifferentiable.

\section{Conclusions}

Departing from the main result in \citep{Tarasov}, we endeavoured
to study the case where the fractional derivative $D_{x}^{\alpha}$
contains locally Hölder functions of exponent $\alpha$ in its domain,
and satisfies the Leibniz rule. Hadamard's representation theorem
for $f\in C^{2}$ around an arbitrary point yields a fractional chain
rule for such $D^{\alpha}$ applied to $f\circ w$, where $w$ is
locally Hölder of exponent $\alpha$, cf. eq (\ref{eq:chainrule}).
As well, we have seen in (\ref{eq:13}) a formula in the case where
$w$ is of class $C^{1}$ and $f$ is Hölder-continuous.

Our main goal here was to reassess a generalization of the Main Theorem
in \cite{Tarasov}. In so doing, we derived a fractional chain rule
for a fractional derivative of order $0<\alpha\leq1$ whose domain
includes $C^{1}$ or even locally Hölder-continuous functions (which
describe coarse-grained media), provided the Leibniz rule holds. This
in turn implies that Leibniz role holds for some alternative definitions
of fractional derivatives \citep{Jumarie,Kolwankar- Nowhere,Adda e Cresson,Kolwankar -smoofthless}.

\textbf{\bigskip{}
 } \textbf{Acknowledgment:} J. A. Helayël-Neto and J. J. Ramón-Marí
are gratefully acknowledged for discussions at all stages of this
work.

\textbf{\bigskip{}
 }

\end{document}